\documentclass[12pt,letterpaper,bibliography=totoc]{scrartcl}

\usepackage{amsmath}
\usepackage{amsthm}
\usepackage{amssymb}
\usepackage{mathtools}
\usepackage{microtype}
\usepackage{xcolor}
\definecolor{Myblue}{rgb}{0,0,0.6}
\usepackage[colorlinks,citecolor=Myblue,linkcolor=Myblue,urlcolor=Myblue,pdfpagemode=UseNone]{hyperref}
\usepackage[totalwidth=420pt,totalheight=590pt]{geometry}
\usepackage[inline,shortlabels]{enumitem}

\allowdisplaybreaks

\setlist[itemize]{leftmargin=2em,itemsep=0.25em,topsep=0.35em}
\setlist[enumerate]{leftmargin=2.2em,itemsep=0.25em,topsep=0.35em}

\theoremstyle{definition}
\newtheorem{theorem}{Theorem}[section]
\newtheorem{proposition}[theorem]{Proposition}
\newtheorem{lemma}[theorem]{Lemma}
\newtheorem{corollary}[theorem]{Corollary}
\newtheorem{definition}[theorem]{Definition}
\newtheorem{remark}[theorem]{Remark}
\numberwithin{equation}{section}
\numberwithin{figure}{section}

\newcommand{\M}{\mathcal M}
\newcommand{\B}{\mathcal B}
\newcommand{\N}{\mathcal N}
\newcommand{\Prym}{\operatorname{Prym}}
\newcommand{\Jac}{\operatorname{Jac}}
\newcommand{\Vol}{\operatorname{Vol}}
\newcommand{\Hit}{\operatorname{Hit}}
\newcommand{\id}{\operatorname{id}}
\newcommand{\End}{\operatorname{End}}
\newcommand{\tr}{\operatorname{tr}}
\newcommand{\Tr}{\operatorname{Tr}}
\newcommand{\dd}{\mathrm d}
\newcommand{\ii}{\mathrm i}
\newcommand{\e}{\mathrm e}
\newcommand{\sfm}{\mathrm{sf}}
\newcommand{\orb}{\mathrm{orb}}
\newcommand{\cH}{C_H}
\newcommand{\cball}{C_{\mathrm{ball}}}
\newcommand{\Eul}{\operatorname{Eul}}
\newcommand{\Coeff}{\operatorname{Coeff}}
\newcommand{\Res}{\operatorname*{Res}}
\newcommand{\Vball}{\mathcal V}
\newcommand{\qrad}{Q}

\title{On the Asymptotics of the Volume of Hitchin Moduli Spaces}

\author{Shiyu Cao\\[-1mm]
  \normalsize Chern Institute of Mathematics, Nankai University\\
  \normalsize\href{mailto:shiyucao@126.com}{shiyucao@126.com}}
\date{September 2026}

\begin{document}

\maketitle

\begin{abstract}
Let $X$ be a compact Riemann surface of genus $g\geq2$, and let $\M$ be the
moduli space of rank-two trace-free Higgs bundles with fixed determinant of
odd degree. For the normalization of the Hitchin metric used in this paper,
we prove that the volume of geodesic ball in the Hitchin moduli space is given by
$$
\Vol_{g_{L^2}}B_{L^2}(p,R)
=\frac{2^{4g-3}\pi^{9g-9}}{(3g-3)!}\,R^{6g-6}+o(R^{6g-6})
$$
for every $p\in\M$. We also determine the leading asymptotics of Hamiltonian
sublevel volumes and exponentially weighted volumes. The proof combines
homogeneity of the hyperkähler volume form with metric asymptotics on the
regular Hitchin locus. Symplectic reduction and the Prym polarization
evaluate the coefficient, which is independently recovered by equivariant
localization.
\end{abstract}

\thispagestyle{empty}
\newpage
\setcounter{tocdepth}{2}
\tableofcontents

\section{Introduction and main results}

Hitchin introduced the self-duality equations on a compact Riemann surface
and showed that their moduli space admits a complex-geometric description in
terms of stable Higgs bundles \cite{Hitchin1987}.  In the fixed-determinant
rank-two case of odd degree, the resulting moduli space is smooth and carries
a complete hyperkähler metric.  In its Dolbeault complex structure, the
Hitchin map makes a dense open subset into an algebraically completely
integrable system whose fibers are Prym torsors associated with smooth
spectral curves; see \cite{BNR1989,Hausel1998,MSWW2019}.

The hyperkähler metric obtained from the gauge-theoretic quotient is the
Hitchin metric, or $L^2$ metric.  On the regular locus of the Hitchin
fibration, the special Kähler metric on the base and the Gauss--Manin
connection determine a second hyperkähler metric, the semiflat metric.  It is
flat on each torus fiber and is induced by the natural $L^2$ metric on
harmonic one-forms of the spectral curve \cite{Freed1999,MSWW2019}.  The
analytic construction of regular ends begins in \cite{MSWW2016}, and the
semiflat metric is the asymptotic model for the Hitchin metric along regular
scaling rays.  Exponential comparison in degree zero is proved by Mochizuki,
and the covering-and-twisting construction extends the comparison to
nonzero degree \cite{MochizukiMetric,MochizukiNonzeroDegree}.

This paper studies the volume growth of the complete metric $g_{L^2}$.  The
real dimension of the moduli space is $12g-12$, while the volume of a large
metric ball grows with exponent $6g-6$, the real dimension of the Hitchin
base.  The expanding base directions and the dual torus directions in the
semiflat metric produce this lower exponent.  We determine the leading
coefficient, identify it by symplectic reduction and the Prym polarization,
and recover it from the equivariant formulas of
Moore--Nekrasov--Shatashvili and Chiarello--Hausel--Szenes
\cite{MNS2000,CHS2020}.

Throughout, $X$ is a compact Riemann surface of genus $g\geq2$, $K_X$ is its
canonical bundle, and
$$
k:=3g-3.
$$

\begin{definition}\label{def:higgs-bundle}
A rank-two Higgs bundle on $X$ is a pair $(E,\Phi)$ consisting of a
holomorphic vector bundle $E\to X$ of rank two and a Higgs field
$$
\Phi\in H^0\bigl(X,\End(E)\otimes K_X\bigr).
$$
Fix a holomorphic line bundle $\Lambda$ of odd degree.  The pair has fixed
determinant $\Lambda$ and is trace-free if
$$
\det E\simeq\Lambda,
\qquad
\tr\Phi=0.
$$
It is stable if every nonzero proper holomorphic subbundle $F\subset E$ with
$\Phi(F)\subset F\otimes K_X$ satisfies $\mu(F)<\mu(E)$, where
$\mu(F)=\deg(F)/\operatorname{rk}(F)$.  We denote the corresponding moduli
space by
$$
\M=\M_{\Lambda}
:=
\left\{
(E,\Phi):\det E\simeq\Lambda,\ \tr\Phi=0,\ (E,\Phi)\text{ stable}
\right\}/\cong.
$$
Since $\gcd(2,\deg\Lambda)=1$, semistability and stability coincide.
\end{definition}

\begin{definition}
\label{def:L2-metric}
Fix a Hermitian metric on $\Lambda$.  By Hitchin's correspondence, every
stable pair $(E,\Phi)$ admits a unique Hermitian metric $h$ inducing the
chosen determinant metric and satisfying the trace-free Hitchin equation
$$
\bigl(F_{D_h}+[\Phi,\Phi^{\dagger_h}]\bigr)_0=0,
$$
where $D_h$ is the Chern connection and $\dagger_h$ denotes the Hermitian
adjoint \cite{Hitchin1987}.  The induced metric on $\End_0(E)$ is harmonic
for the adjoint Higgs bundle $(\End_0(E),\operatorname{ad}\Phi)$.  A tangent
class is represented by an $h$-harmonic $\End_0(E)$-valued one-form
$\tau=\tau^{1,0}+\tau^{0,1}$.  Following
\cite[Section~4.1.2]{MochizukiNonzeroDegree}, set
$$
\langle\tau_1,\tau_2\rangle_{L^2}
:=
2\ii\int_X\tr\left(
\tau_1^{1,0}\wedge(\tau_2^{1,0})^\dagger
-
\tau_1^{0,1}\wedge(\tau_2^{0,1})^\dagger
\right).
$$
The harmonic representative is unique, so this formula defines a Hermitian
pairing on $T\M$.  Its real part is the Hitchin metric $g_{L^2}$.  We write $\omega_I$ for its Kähler
form in the Dolbeault complex structure and
$\Omega_I=\omega_J+\ii\omega_K$ for the holomorphic symplectic form.  The
circle action and positive real scaling are
$$
\e^{\ii\vartheta}\cdot(E,\Phi)=(E,\e^{\ii\vartheta}\Phi),
\qquad
\delta_t(E,\Phi)=(E,t\Phi),
$$
and we set
$$
H(E,\Phi):=\frac12\|\Phi\|_{L^2}^2,
\qquad
\dd H=-\iota_K\omega_I,
$$
where $K$ is the period-$2\pi$ circle generator.  The generator of
$\delta_t$ with respect to $\log t$ is denoted by $\xi$.
\end{definition}

\begin{definition}
\label{def:semiflat-metric}
The rank-two Hitchin fibration is
$$
\Hit:\M\longrightarrow\B:=H^0(X,K_X^2),
\qquad
(E,\Phi)\longmapsto-\det\Phi.
$$
Let $\Delta\subset\B$ be the discriminant and set
$$
\B':=\B\setminus\Delta
=
\{q\in\B:q\text{ has only simple zeros}\},
\qquad
\M':=\Hit^{-1}(\B').
$$
For $q\in\B'$, write $S_q$ for the smooth spectral curve. The natural flat connection on
the Prym torsors gives a horizontal--vertical splitting of $T\M'$. Under
the spectral identification, its horizontal and vertical summands correspond
to anti-invariant harmonic forms of types $(1,0)$ and $(0,1)$, respectively.
The semiflat metric $g_{\sfm}$ is the real part of the pairing
$$
2\ii\int_{S_q}\left(
\sigma_1^{1,0}\wedge\overline{\sigma_2^{1,0}}
-\sigma_1^{0,1}\wedge\overline{\sigma_2^{0,1}}
\right).
$$
We use the connection and normalization of
\cite[Sections~3.1.2 and~4.3]{MochizukiNonzeroDegree}.
The induced metric on $\B'$ is denoted by $g_{\mathrm{sK}}$.
We write $\omega_{\sfm}$ for the $I$-Kähler form of $g_{\sfm}$.
\end{definition}

\begin{definition}
\label{def:volume-functions}
Since $\dim_{\mathbb C}\M=2k$, the Riemannian volume form is
$$
\dd V_{L^2}:=\frac{\omega_I^{2k}}{(2k)!}.
$$
Write $d_{L^2}$ for the distance induced by $g_{L^2}$.  For
$p_0\in\M$, $s\geq0$, $R>0$, and $\varepsilon>0$, define
$$
\begin{aligned}
A(s)&:=\Vol_{g_{L^2}}\{H\leq s\},\\
B_{L^2}(p_0,R)&:=\{x\in\M:d_{L^2}(p_0,x)<R\},\\
\Vball_{p_0}(R)&:=\Vol_{g_{L^2}}B_{L^2}(p_0,R),\\
V_H(\varepsilon)&:=\int_{\M}\e^{-\varepsilon H}\,\dd V_{L^2}.
\end{aligned}
$$
The constants appearing in the main asymptotics are
\begin{equation}\label{eq:explicit-leading-constants}
\cH:=\frac{2^{7g-6}\pi^{9g-9}}{(3g-3)!},
\qquad
\cball:=2^{-k}\cH
=
\frac{2^{4g-3}\pi^{9g-9}}{(3g-3)!}.
\end{equation}
\end{definition}

\begin{theorem}\label{thm:main}
With the notation of Definition~\ref{def:volume-functions},
\begin{align}
A(s)&=\cH s^k+O(s^{k-1})
&& (s\to\infty),\label{eq:sublevel-main}\\
\Vball_{p_0}(R)&=\cball R^{2k}+o(R^{2k})
&& (R\to\infty),\label{eq:ball-main}\\
V_H(\varepsilon)&=k!\cH\,\varepsilon^{-k}
+O(\varepsilon^{-k+1})
&& (\varepsilon\downarrow0).\label{eq:weighted-main}
\end{align}
In particular, the volume-growth exponent is $2k=6g-6$ and the leading
coefficient is independent of the center $p_0$.
\end{theorem}

The scaling law for the holomorphic symplectic form gives an exact radial
decomposition of the hyperkähler volume form above a sufficiently large
regular value of $H$. The semiflat comparison determines the growth of $H$
along regular scaling rays. The remaining part of the level set has measure
zero, and a determinant inequality supplies a uniform bound for dominated
convergence. The Duistermaat--Heckman formula expresses the resulting
coefficient as an intersection number, which is evaluated using the Prym
polarization. Compact exhaustion of the regular level set gives the lower
bound for metric balls, while the Lipschitz estimate for $\sqrt{2H}$ gives
the upper bound. These steps require metric asymptotics uniform on compact
subsets of the regular locus. The final section recovers the coefficient by
equivariant localization and relates it to the formulas of
\cite{MNS2000,CHS2020}.

\subsection{Fixed-determinant normalization}

For the remainder of the paper, take the moduli space of
Definition~\ref{def:higgs-bundle} with $\deg\Lambda=1$ and write
$\M=\M_{\Lambda}$.  Then
$$
\dim_{\mathbb C}\M=2k=6g-6,
\qquad
\dim_{\mathbb R}\M=4k=12g-12.
$$
Every odd determinant degree reduces isometrically to this case.  Indeed, if
$\deg\Lambda$ is odd, choose a line bundle $N$ with
$$
\deg N=\frac{\deg\Lambda-1}{2}.
$$
Then
$$
(E,\Phi)\longmapsto(E\otimes N^{-1},\Phi)
$$
changes the determinant degree to one.  Tensoring identifies
$\End_0(E)$ with $\End_0(E\otimes N^{-1})$, hence also the fixed-determinant
trace-free deformation complexes and their $L^2$ pairings.  The normalization of the universal cohomology class is recorded in
\eqref{eq:hausel-normalization-dictionary}.

Fix an auxiliary conformal metric on $X$ and use it to define pointwise norms
and the area density $\dd A$. For a complex $(1,0)$-form $\eta$,
our tensor norms satisfy
$$
\ii\eta\wedge\overline\eta=|\eta|^2\dd A.
$$
The densities
$$
|q|\,\dd A,
\qquad
\frac{|\dot q|^2}{|q|}\,\dd A
$$
are conformally invariant.  Throughout, $\tr$ is the ordinary trace in the
defining rank-two representation and $\dagger$ denotes the adjoint with
respect to the harmonic metric.

\subsection{The regular Hitchin system and special Kähler geometry}

We now record the rank-two spectral conventions used below, following
\cite[Sections 2.2--2.4]{MSWW2019}.  A special Kähler metric is a Kähler
metric together with a flat, torsion-free, symplectic connection $\nabla$
satisfying $\dd^\nabla I=0$; see \cite{Freed1999}.  The restriction
$\Hit:\M'\to\B'$ is an algebraically completely integrable system with
compact complex Lagrangian fibers.  For $q\in\B'$, the spectral curve is
$$
S_q
=
\{\lambda\in K_X:\lambda^2=q\}
\xrightarrow{\ p_q\ }
X.
$$
Let $\sigma_q$ denote its deck involution.  The curve $S_q$ is a smooth
double cover of genus $4g-3$, and the Hitchin fiber over $q$ is a torsor for
$$
\Prym(S_q/X)
=
\ker\bigl(\operatorname{Nm}:\Jac(S_q)\to\Jac(X)\bigr)^0.
$$
After choosing a local horizontal section over a sufficiently small open
set $U\subset\B'$, one may identify the fibers of $\Hit^{-1}(U)$ with the
corresponding Prym varieties.  Thus
$$
\dim_{\mathbb C}\Prym(S_q/X)=3g-3=k.
$$

Let $\lambda$ be the Seiberg--Witten differential, obtained by restricting
the tautological one-form on $K_X$ to $S_q$. Then
$$
\sigma_q^*\lambda=-\lambda,
\qquad \lambda^2=q.
$$
The Gauss--Manin derivative of $\lambda$ in the direction $\dot q$ is
$$
\dot\lambda=\frac{\dot q}{2\lambda}.
$$
This is a holomorphic anti-invariant one-form on $S_q$, including at the
ramification points. Definition~\ref{def:semiflat-metric} therefore gives
\begin{equation}\label{eq:special-kahler-L2}
g_{\mathrm{sK}}(\dot q,\dot q)
=2\ii\int_{S_q}\dot\lambda\wedge\overline{\dot\lambda}
=\frac12\int_{S_q}\left|\frac{\dot q}{\lambda}\right|^2\dd A
=\int_X\frac{|\dot q|^2}{|q|}\,\dd A.
\end{equation}
The expression in \cite[Section~2.3]{MSWW2019} is multiplied by four here
to agree with the $L^2$ pairing in Definition~\ref{def:L2-metric}.

Higgs-field scaling induces $q\mapsto t^2q$. Spectral dilation
$(x,\lambda)\mapsto(x,t\lambda)$ identifies $S_q$ with $S_{t^2q}$ and
preserves the Gauss--Manin local system. In particular,
\begin{equation}\label{eq:special-coordinate-homogeneity}
g_{\mathrm{sK},t^2q}(t^2\dot q,t^2\dot q)
=t^2g_{\mathrm{sK},q}(\dot q,\dot q).
\end{equation}
The flat connection on the Prym torsors supplies the horizontal--vertical
splitting used in Definition~\ref{def:semiflat-metric}. A local horizontal
section identifies the torsors with their associated Prym varieties.
Spectral dilation transports both the spectral line bundle and the pullback
of $K_X$ as constant families, so the scaling orbit is horizontal in the
sense of \cite[Definition~3.3]{MochizukiNonzeroDegree}. The vertical metric
is unchanged by this transport, while the horizontal metric scales as in
\eqref{eq:special-coordinate-homogeneity}.

The real $2k$-dimensional horizontal space expands under scaling, while the
vertical metric is unchanged. This accounts for the semiflat volume-growth
exponent. The radial argument below gives the exact coefficient for the
Hitchin metric.

\subsection*{Acknowledgments}
The author would like to express his gratitude to Profs. Kefeng Liu and Huitao Feng for the guidance over the years. And he would like to thank Prof. Qiongling Li for her encouragement. The idea of this paper originated from the course at Nankai concerning analytic localization given by Prof. Feng  in 2022. After completing the main part of this paper, the author found the work of \cite{CHS2020} by ChatGPT, through which we obtained a parallel calculation , which gave rise to the same result. The author decided to keep it as an independent check.
\section{Radial geometry and metric asymptotics}\label{sec:radial-geometry}

\subsection{The circle action and a global Lipschitz estimate}

For unitary connection variations $a,b$, the convention of
Definition~\ref{def:L2-metric} gives
$$
\omega_I((a,0),(b,0))=-\int_X\tr(a\wedge b).
$$
On the Higgs-field summand,
$$
\langle\varphi,\psi\rangle_{L^2}
=2\ii\int_X\tr(\varphi\wedge\psi^\dagger).
$$

\begin{lemma}\label{lem:moment-map-normalization}
The infinitesimal generator of the circle action on configuration space is
$$
\widetilde K=(0,\ii\Phi),
\qquad
\|\widetilde K\|_{L^2}^2=2H.
$$
The induced vector field on the moduli space satisfies
\begin{equation}\label{eq:gradient-generator}
\xi=-IK=\nabla H.
\end{equation}
\end{lemma}

\begin{proof}
For a tangent representative $(a,\varphi)$,
$$
\dd H(a,\varphi)
=\operatorname{Re}\langle\Phi,\varphi\rangle_{L^2}.
$$
On the Higgs-field summand, $I\widetilde K=(0,-\Phi)$.  Hence
$$
\omega_I(\widetilde K,(a,\varphi))
=g_{L^2}(I\widetilde K,(a,\varphi))
=-\operatorname{Re}\langle\Phi,\varphi\rangle_{L^2}
=-\dd H(a,\varphi).
$$
Orthogonal projection of $\widetilde K$ to the Coulomb slice preserves its
pairing with horizontal tangent vectors, so the identity descends to $\M$.
The positive real subgroup of the complexified circle action is generated by
$-IK$, which is the $\log t$ generator of $\delta_t$.
\end{proof}

Since $K$ is the orthogonal projection of $\widetilde K$ onto the Coulomb
slice, equation~\eqref{eq:gradient-generator} gives
$$
|\nabla H|^2=|K|^2\leq2H.
$$
Consequently, $\sqrt{2H}$ is globally $1$-Lipschitz.  On
$\{H>0\}$,
$$
|\dd\sqrt{2H}|
=\frac{|\dd H|}{\sqrt{2H}}
\leq1.
$$
Applying the same estimate to $\sqrt{2H+\varepsilon}$ near $H^{-1}(0)$ and
letting $\varepsilon\downarrow0$ yields
\begin{equation}\label{eq:lipschitz-global}
\left|\sqrt{2H(x)}-\sqrt{2H(y)}\right|
\leq d_{L^2}(x,y).
\end{equation}

\subsection{Homogeneity of the holomorphic symplectic form}

The holomorphic symplectic form is induced by the Serre pairing between
bundle and Higgs-field variations. Scaling fixes the former and multiplies
the latter by $t$, so
$$
\delta_t^*\Omega_I=t\Omega_I.
$$
With $\Omega_I=\omega_J+\ii\omega_K$ and $\dim_{\mathbb C}\M=2k$, the
hyperkähler volume form is
$$
\dd V_{L^2}
=\frac{\Omega_I^k\wedge\overline{\Omega_I}^{\,k}}{4^k(k!)^2}.
$$
It follows that
\begin{equation}\label{eq:volume-weight}
\delta_t^*\dd V_{L^2}=t^{2k}\dd V_{L^2}.
\end{equation}

\subsection{Radial parametrization above the largest critical value}

\begin{definition}
\label{def:high-level-set}
The nilpotent cone is
$$
\N:=\Hit^{-1}(0).
$$
It is compact by properness of the Hitchin map.  Fix a regular value $s_0$
above every critical value of $H$ and satisfying
$$
s_0>\max_{\N}H.
$$
Set
$$
P:=H^{-1}(s_0),
\qquad
P':=P\cap\M'.
$$
Then $P$ is a compact smooth hypersurface and $P\cap\N=\varnothing$.
\end{definition}

\begin{lemma}\label{lem:radial-parametrization}
With $s_0$ and $P$ as in Definition~\ref{def:high-level-set}, every orbit
of the positive real subgroup that meets $\{H\geq s_0\}$ intersects $P$
exactly once, and
$$
F:P\times[1,\infty)\longrightarrow\{H\geq s_0\},
\qquad
F(p,t)=\delta_t p
$$
is a diffeomorphism.
\end{lemma}

\begin{proof}
For $x\in\M$, consider the algebraic orbit map
$$
f_x:\mathbb C^*\longrightarrow\M,
\qquad
t\longmapsto\delta_t x.
$$
Its composition with the Hitchin map is
$$
\Hit\circ f_x(t)=t^2\Hit(x),
$$
which extends holomorphically across $t=0$.  Since
$\Hit:\M\to\B$ is proper, the valuative criterion gives a unique extension
of $f_x$ across $t=0$.  Its value at $0$ lies in
$\N=\Hit^{-1}(0)$.  Along an orbit of the positive real subgroup,
$$
\frac{\dd}{\dd\log t}H(\delta_t x)
=
\dd H(\xi)
=
|\xi|_{g_{L^2}}^2.
$$
There are no critical points on $\{H\geq s_0\}$, so this function is
strictly increasing there.

If $p\in P$, then $\Hit(p)\neq0$ and
$$
\Hit(\delta_t p)=t^2\Hit(p).
$$
Thus $\delta_t p$ leaves every compact subset as $t\to\infty$. Properness of
$H$ gives
$$
H(\delta_t p)\longrightarrow\infty.
$$
The backward limit lies in $\N$, where $H<s_0$. Strict monotonicity and the
intermediate value theorem give a unique intersection. Applying the implicit
function theorem to
$$
(p,t)\longmapsto H(\delta_t p)
$$
and using $\partial_{\log t}H=|\xi|^2>0$ shows that the intersection
depends smoothly on the orbit.  The resulting inverse to $F$ is smooth, so
$F$ is a diffeomorphism.
\end{proof}

\begin{definition}\label{def:radial-measure}
Orient $P$ by the outward normal to $\{H\leq s_0\}$ and set
$$
\dd\mu_P
:=
\left.\iota_\xi\dd V_{L^2}\right|_P.
$$
This is a positive smooth density on $P$; its associated measure is denoted
by $\mu_P$.
\end{definition}

Since $F_*\partial_t=\xi/t$, equation~\eqref{eq:volume-weight} gives
\begin{equation}\label{eq:radial-volume}
F^*\dd V_{L^2}
=
t^{2k-1}\dd t\wedge\dd\mu_P.
\end{equation}

\subsection{Full measure of the regular locus and a uniform quadratic lower bound}

The complement $\M\setminus\M'=\Hit^{-1}(\Delta)$ is a proper complex
analytic subset of $\M$, hence has real codimension at least two.  Its
intersection with the smooth real hypersurface $P$ has zero measure with
respect to the smooth density $\dd\mu_P$.  Therefore
\begin{equation}\label{eq:singular-measure-zero}
\mu_P(P\setminus P')=0.
\end{equation}

The pointwise determinant inequality gives
\begin{equation}\label{eq:determinant-bound}
H(E,\Phi)\geq2\int_X|\det\Phi|\,\dd A.
\end{equation}
Indeed, write $\Phi=\varphi\zeta$ in unitary frames, with $|\zeta|=1$,
and let $s_1,s_2$ be the singular values of $\varphi$. Then
$$
2|\det\varphi|=2s_1s_2\leq s_1^2+s_2^2
=\tr(\varphi\varphi^\dagger).
$$
Since $H=\int_X|\Phi|^2\,\dd A$ with our tensor norms, integration proves
\eqref{eq:determinant-bound}.

For $p=[E,\Phi_p]\in P$, define
$$
q_p:=-\det\Phi_p,
\qquad
\qrad(p):=\int_X|q_p|\,\dd A.
$$
The function $\qrad$ is continuous and positive on $P$, because
$P\cap\N=\varnothing$. Set
$$
m_0:=2\min_{p\in P}\qrad(p)>0.
$$
Since $\det(t\Phi_p)=t^2\det\Phi_p$, \eqref{eq:determinant-bound} yields
\begin{equation}\label{eq:uniform-quadratic-lower}
H(\delta_t p)\geq m_0t^2
\qquad (p\in P,\ t\geq1).
\end{equation}

\subsection{The semiflat norm of the scaling vector field}

On the regular locus, the semiflat metric provides the model for the radial norm.

\begin{lemma}\label{lem:semiflat-scaling}
For $p\in P'$ and $t\geq1$,
\begin{equation}\label{eq:semiflat-radial-norm}
g_{\sfm}(\xi,\xi)_{\delta_t p}
=
4t^2\qrad(p).
\end{equation}
\end{lemma}

\begin{proof}
Fiberwise dilation
in $K_X=T^*X$,
$$
S_{q_p}\longrightarrow S_{t^2q_p},
\qquad
(x,\lambda)\longmapsto(x,t\lambda),
$$
commutes with the spectral involutions.  It induces Gauss--Manin flat
transport on anti-invariant integral homology and identifies the
corresponding Prym torsors.  Under the BNR correspondence \cite{BNR1989},
the spectral line bundle representing $\delta_t p$ is the transport of the
line bundle representing $p$.  Thus the scaling orbit is horizontal for the
semiflat connection.

At $\delta_t p$, the base point is $t^2q_p$ and the base component of $\xi$ is
$2t^2q_p$. Equation \eqref{eq:special-kahler-L2} gives
$$
\begin{aligned}
g_{\sfm}(\xi,\xi)_{\delta_t p}
&=
g_{\mathrm{sK}}(2t^2q_p,2t^2q_p)_{t^2q_p}
\\
&=
\int_X
\frac{|2t^2q_p|^2}{|t^2q_p|}\,\dd A
\\
&=
4t^2\int_X|q_p|\,\dd A.
\end{aligned}
$$
\end{proof}

\subsection{Reduction of the metric comparison to degree zero}

We transfer the calculations to the Hitchin metric using the following comparison.

\begin{proposition}\label{prop:compact-comparison}
For every compact subset $\mathcal K\Subset\M'$, there are constants $C,c>0$
such that
\begin{equation}\label{eq:compact-comparison}
\left|
g_{\sfm}^{-1}g_{L^2}-\id
\right|_{g_{\sfm},\,\delta_t p}
\leq C\e^{-ct}
\end{equation}
for all $p\in\mathcal K$ and $t\geq1$.  Here
$g_{\sfm}^{-1}g_{L^2}$ is regarded as a positive
$g_{\sfm}$-self-adjoint endomorphism, and the norm is its operator norm with
respect to $g_{\sfm}$.
\end{proposition}

\begin{proof}
Choose a connected unramified cyclic Galois cover of even degree $d_\psi$,
$$
\psi:Y\longrightarrow X
$$
and a line bundle $L\to Y$ with $\deg L=d_\psi/2$. On $\M'$, define
$$
F_{\psi,L}(E,\Phi)
=
\bigl(\psi^*E\otimes L^{-1},\psi^*\Phi\bigr).
$$
Since $\deg E=1$ and $\deg\psi=d_\psi$, one has
$$
\deg(\psi^*E\otimes L^{-1})=d_\psi-2\deg L=0.
$$
The map commutes with scaling. Mochizuki's covering and twisting formulas give
$$
F_{\psi,L}^*g_{L^2,Y}
=
d_\psi g_{L^2,X},
\qquad
F_{\psi,L}^*g_{\sfm,Y}
=
d_\psi g_{\sfm,X};
$$
see
\cite[Lemmas 4.2 and 4.12 and Propositions 4.9 and 4.13]
{MochizukiNonzeroDegree}.

Because $\psi$ is unramified, the spectral curve of $\psi^*\Phi$ is the
smooth base change of the spectral curve of $\Phi$.  The simple zeros of
$q=-\det\Phi$ pull back to nonempty simple branching over the connected
curve $Y$, so this double cover is connected.  By the BNR correspondence, connectedness of the smooth spectral
curve excludes a proper Higgs-invariant line subbundle.  Hence the
pulled-back Higgs bundle is stable and belongs to the stable smooth locus of
the degree-zero regular moduli space; twisting by $L^{-1}$ preserves these
properties.  Consequently,
$F_{\psi,L}(\mathcal K)$ is a compact subset of that locus. Apply
\cite[Theorem 4.7]{MochizukiMetric} to this compact subset and pull the estimate
back along $F_{\psi,L}$. The common factor $d_\psi$ in the two metrics cancels,
which proves \eqref{eq:compact-comparison}.

Moreover,
$$
\det(\psi^*E\otimes L^{-1})
=
\psi^*\Lambda\otimes L^{-2},
\qquad
\tr(\psi^*\Phi)=0.
$$
Thus $F_{\psi,L}$ maps the present moduli space to the degree-zero
fixed-determinant locus. Its differential preserves the trace-free tangent
spaces, and the degree-zero metric estimate restricts to these spaces.
\end{proof}

\subsection{Asymptotics of \texorpdfstring{$H$}{H} and scaling-orbit length}

Fix a compact subset $\mathcal K\Subset P'$. Equations
\eqref{eq:semiflat-radial-norm} and \eqref{eq:compact-comparison} give,
uniformly for $p\in\mathcal K$,
$$
|\xi|_{g_{L^2}}^2
=
4t^2\qrad(p)\bigl(1+O_{\mathcal K}(\e^{-ct})\bigr).
$$
Since $\xi=\nabla H$,
$$
\frac{\dd}{\dd t}H(\delta_t p)
=
\frac1t|\xi|_{g_{L^2}}^2
=
4t\qrad(p)\bigl(1+O_{\mathcal K}(\e^{-ct})\bigr).
$$
Integrating from $1$ to $t$ and using
$\int_1^\infty u\e^{-cu}\,\dd u<\infty$ gives
\begin{equation}\label{eq:H-ray-asymptotic}
H(\delta_t p)
=
2\qrad(p)t^2+O_{\mathcal K}(1).
\end{equation}

The length of the same scaling path is
$$
\begin{aligned}
L(p,t)
&=
\int_1^t
\left|\frac{\dd}{\dd u}\delta_u p\right|_{g_{L^2}}\dd u
\\
&=
\int_1^t\frac{|\xi|_{g_{L^2}}}{u}\,\dd u
\\
&=
2\int_1^t\sqrt{\qrad(p)}
\bigl(1+O_{\mathcal K}(\e^{-cu})\bigr)\,\dd u
\\
&=
2t\sqrt{\qrad(p)}+O_{\mathcal K}(1).
\end{aligned}
$$
Since $\qrad$ has a positive lower bound on $\mathcal K$,
\eqref{eq:H-ray-asymptotic} gives
$$
\sqrt{2H(\delta_t p)}
=
2t\sqrt{\qrad(p)}+O_{\mathcal K}(t^{-1}).
$$
Therefore
\begin{equation}\label{eq:length-ray-asymptotic}
L(p,t)
=
\sqrt{2H(\delta_t p)}+O_{\mathcal K}(1).
\end{equation}
These estimates are uniform on each fixed compact subset of $P'$.
\section{Hamiltonian volumes and the Prym coefficient}
\label{sec:hamiltonian-volumes}

\subsection{Hamiltonian sublevel volumes}

For $s\geq s_0$ and $p\in P$, let $T_s(p)\geq1$ be the unique number satisfying
$$
H(\delta_{T_s(p)}p)=s.
$$
Uniqueness follows from strict monotonicity of $H$ along each orbit of the
positive real subgroup. The exact radial formula \eqref{eq:radial-volume}
gives
\begin{equation}\label{eq:A-radial-exact}
\begin{aligned}
A(s)-A(s_0)
&=
\int_P\int_1^{T_s(p)}t^{2k-1}\,\dd t\,\dd\mu_P(p)
\\
&=
\frac1{2k}\int_P\bigl(T_s(p)^{2k}-1\bigr)\,\dd\mu_P(p).
\end{aligned}
\end{equation}

If $p\in P'$, apply \eqref{eq:H-ray-asymptotic} on a compact neighborhood
$\mathcal U\Subset P'$ of $p$. This gives the pointwise limit
$$
\frac{T_s(p)^2}{s}
\longrightarrow
\frac{1}{2\qrad(p)}.
$$
The uniform estimate \eqref{eq:uniform-quadratic-lower} gives
$$
s=H(\delta_{T_s(p)}p)
\geq m_0T_s(p)^2,
$$
and hence
\begin{equation}\label{eq:T-domination}
0\leq\frac{T_s(p)^2}{s}\leq m_0^{-1}
\end{equation}
for all $p\in P$ and $s\geq s_0$. Divide \eqref{eq:A-radial-exact} by
$s^k$. Equations \eqref{eq:singular-measure-zero} and
\eqref{eq:T-domination}, together with dominated convergence, yield
$$
\lim_{s\to\infty}\frac{A(s)}{s^k}
=
\frac1{2k}
\int_{P'}
\left(\frac{1}{2\qrad(p)}\right)^k\,\dd\mu_P(p).
$$
The radial expression for the leading coefficient is therefore
\begin{equation}\label{eq:CH-radial}
\cH
=
\frac1{2k}
\int_{P'}
\left(\frac{1}{2\qrad(p)}\right)^k\,\dd\mu_P(p),
\end{equation}
where the equality with the explicit constant in
\eqref{eq:explicit-leading-constants} is established below. By the
definition of $m_0$,
$$
\qrad(p)\geq m_0/2
\qquad (p\in P').
$$
Since $P$ is compact and $\mu_P(P)<\infty$, the integral is finite.  It is
strictly positive because $P'$ has full measure in $P$.  Therefore
\begin{equation}\label{eq:A-o}
A(s)=\cH s^k+o(s^k).
\end{equation}

For a compact subset $\mathcal K\Subset P'$, set
$$
\cH(\mathcal K)
:=
\frac1{2k}
\int_{\mathcal K}
\left(\frac{1}{2\qrad(p)}\right)^k\,\dd\mu_P(p).
$$
If $\mathcal K_j\Subset P'$ is an increasing compact exhaustion, monotone
convergence gives
\begin{equation}\label{eq:CHK-exhaust}
\cH(\mathcal K_j)\longrightarrow\cH.
\end{equation}

The global bound \eqref{eq:uniform-quadratic-lower} controls the radial
cutoff even on $P\setminus P'$. Its measure-zero contribution therefore
does not affect the limit. We next compute the radial coefficient by
symplectic reduction.

\subsection{Symplectic reduction and polynomiality}

\begin{definition}
\label{def:reduced-data}
For $s\geq s_0$, set
$$
P_s:=H^{-1}(s),
\qquad
Z_s:=P_s/S^1,
$$
and let $\pi_s:P_s\to Z_s$ be the quotient map.  Above the largest critical
value, Hausel's biholomorphic identifications allow us to regard the $Z_s$ as
a single compact Kähler orbifold $Z$; the reduced Kähler form at level $s$ is
written $\omega_s$.  All integrals over $Z$ are orbifold integrals.
\end{definition}

Choose an $S^1$-invariant connection one-form $\vartheta$ with
$\vartheta(K)=1$ and $\int_{S^1}\vartheta=2\pi$.  With the right circle
action on $P_s$, define
$$
L_Z:=(P_s\times\mathbb C)/{\sim},
\qquad
(p,z)\sim(p\cdot\e^{\ii\theta},\e^{\ii\theta}z),
$$
and set
\begin{equation}\label{eq:euler-class-definition}
e:=c_1(L_Z)
=\left[\frac{\ii F_\nabla}{2\pi}\right]
=\left[\frac{\dd\vartheta}{2\pi}\right]
\in H^2_{\orb}(Z;\mathbb R).
\end{equation}

Trivialize the end by the flow of $\nabla H/|\nabla H|^2$. On $P_s$,
choose the connection
$$
\vartheta_s:=\left.\iota_{\nabla H/|\nabla H|^2}\omega_I\right|_{P_s}.
$$
Then $\vartheta_s(K)=1$, and the Hamiltonian normal form is
\begin{equation}\label{eq:normal-form}
\omega_I=\pi_s^*\omega_s+\dd H\wedge\vartheta_s.
\end{equation}
Since
$$
(\dd H\wedge\vartheta_s)^2=0,
\qquad
(\pi_s^*\omega_s)^{2k}=0,
$$
we have
$$
\frac{\omega_I^{2k}}{(2k)!}
=
\dd H\wedge\vartheta_s\wedge
\frac{\pi_s^*\omega_s^{2k-1}}{(2k-1)!}.
$$
Integration along the circle fiber gives
\begin{equation}\label{eq:DH-density}
A'(s)
=
2\pi\int_Z
\frac{\omega_s^{2k-1}}{(2k-1)!}.
\end{equation}
Orbifold integration accounts for finite stabilizers, so no additional
division by the stabilizer order is required.

Closedness of $\omega_I$ determines the variation of the reduced class.

\begin{lemma}
With the reduced spaces and the class $e$ from
\eqref{eq:euler-class-definition}, one has
\begin{equation}\label{eq:DH-linear}
[\omega_s]
=
[\omega_{s_0}]+2\pi(s-s_0)e.
\end{equation}
\end{lemma}

\begin{proof}
Taking the exterior derivative of \eqref{eq:normal-form} in the chosen
trivialization gives
$$
\pi_s^*\left(\frac{\partial\omega_s}{\partial s}\right)
=\dd_{P_s}\vartheta_s.
$$
The connections $\vartheta_s$ have period $2\pi$, and their curvatures
descend to representatives of the fixed class $2\pi e$ by
\eqref{eq:euler-class-definition}. Integrating in $s$ proves
\eqref{eq:DH-linear}.
\end{proof}

Let
$$
\mathbb P\B'
:=
\{[q]\in\mathbb P\B:q\in\B'\}
$$
be the regular projective locus.  Since
$\Hit(\delta_t x)=t^2\Hit(x)$, the Hitchin map descends after reduction
to the projectivized map
$$
\bar\chi:Z\longrightarrow\mathbb P\B\simeq\mathbb P^{k-1}.
$$
Hausel proves that
$$
L_Z^{\otimes2}
\simeq
\bar\chi^*\mathcal O_{\mathbb P^{k-1}}(1).
$$
Let
$$
h:=c_1\bigl(\mathcal O_{\mathbb P^{k-1}}(1)\bigr),
\qquad
\int_{\mathbb P^{k-1}}h^{k-1}=1.
$$
Then
\begin{equation}\label{eq:e-hyperplane}
2e=\bar\chi^*h,
\qquad
e^k=0.
\end{equation}

Expanding \eqref{eq:DH-density} using \eqref{eq:DH-linear}, the terms
with at least $k$ factors of $e$ vanish. Thus
$$
A'(s)
=
2\pi
\sum_{j=0}^{k-1}
\frac{(2\pi)^j(s-s_0)^j}{j!(2k-1-j)!}
\int_Z e^j\wedge\omega_{s_0}^{2k-1-j}.
$$
Integration gives
$$
A(s)
=
A(s_0)
+
2\pi
\sum_{j=0}^{k-1}
\frac{(2\pi)^j(s-s_0)^{j+1}}{(j+1)!(2k-1-j)!}
\int_Z e^j\wedge\omega_{s_0}^{2k-1-j}.
$$
Thus $A(s)$ is a polynomial of degree at most $k$ for $s>s_0$. Equation
\eqref{eq:A-o} and the positivity of $\cH$ show that its degree is exactly
$k$. Comparing leading coefficients gives
\begin{equation}\label{eq:CH-intersection}
\cH
=
\frac{(2\pi)^k}{(k!)^2}
\int_Z e^{k-1}\wedge\omega_{s_0}^k.
\end{equation}
The integrand has real degree
$2(k-1)+2k=4k-2=\dim_{\mathbb R}Z$.  This identifies the radial coefficient
in \eqref{eq:CH-radial}, and \eqref{eq:sublevel-main} follows.

The formula is independent of $s_0$, since
$$
\frac{\dd}{\dd s_0}
\int_Z e^{k-1}\wedge\omega_{s_0}^k
=
2\pi k
\int_Z e^k\wedge\omega_{s_0}^{k-1}
=0.
$$

\subsection{Prym expression for the leading coefficient}

\begin{definition}
\label{def:kummer-fiber}
For $[b]\in\mathbb P\B'$, choose a nonzero representative $b\in\B'$ and
let $\Prym_b$ be the corresponding Prym variety.  For $c=t^2$, spectral dilation by $t$ identifies the Prym varieties
associated with $b$ and $cb$. The two choices of $t$ differ by the spectral
involution, which acts as inversion on the Prym. Thus the induced
identification of the Kummer quotients is independent of this choice. We write
$$
K_b:=[\Prym_b/\{\pm1\}]
$$
for the quotient orbifold and
$$
\pi_b:\Prym_b\longrightarrow K_b
$$
for its degree-two orbifold covering.  If $\omega_{s_0}$ is the reduced form,
set
$$
\omega_{\Prym,b}:=\pi_b^*(\omega_{s_0}|_{K_b}).
$$
\end{definition}

By \eqref{eq:e-hyperplane} and fiber integration,
$$
\int_Z e^{k-1}\wedge\omega_{s_0}^k
=
2^{-(k-1)}
\int_{K_b}^{\orb}\omega_{s_0}^k.
$$
Substitution into \eqref{eq:CH-intersection} gives
$$
\cH
=
\frac{2\pi^k}{(k!)^2}
\int_{K_b}^{\orb}\omega_{s_0}^k
=
\frac{2\pi^k}{k!}
\Vol^{\orb}_{\omega_{s_0}}(K_b).
$$
Orbifold integration gives
$$
\Vol^{\orb}_{\omega_{s_0}}(K_b)
=
\frac12\Vol_{\omega_{\Prym,b}}(\Prym_b).
$$
Hence
$$
\cH
=
\frac{\pi^k}{k!}
\Vol_{\omega_{\Prym,b}}(\Prym_b).
$$

We now compare the reduced class with the class on a Hitchin fiber.
Use the notation of Definition~\ref{def:kummer-fiber} and fix
$[b]\in\mathbb P\B'$. Choose a sufficiently small nonzero representative,
still denoted by $b\in\B'$, such that the compact fiber $\Hit^{-1}(b)$ lies
in $\{H<s_0\}$. Such a representative exists: otherwise there would be a
sequence $b_n\to0$ and points $x_n\in\Hit^{-1}(b_n)$ with
$H(x_n)\geq s_0$. Properness of the Hitchin map over a compact neighborhood
of the origin would give a convergent subsequence with limit in the
nilpotent cone, contradicting $s_0>\max_{\N}H$.

Choose a fixed point of the involution defining the Kummer quotient as the
origin of the Prym torsor. The involution then acts as inversion. Denote
the inclusion by
$$
j_b:\Prym_b\longrightarrow\M'.
$$
Let $\Theta_b\in H^2(\Prym_b;\mathbb Z)$ be the restriction to $\Prym_b$ of
the principal polarization class on $\Jac(S_b)$.

\begin{proposition}\label{prop:prym-polarization}
With these choices,
$$
[\omega_{\Prym,b}]
=
[j_b^*\omega_I]
=
4\pi^2\Theta_b.
$$
The polarization has type
$$
(1^{\,2g-3},2^{\,g}),
$$
and
$$
\int_{\Prym_b}\frac{\Theta_b^k}{k!}=2^g.
$$
\end{proposition}

\begin{proof}
We first compare the reduced form with the form on the Hitchin fiber. For each
$L\in\Prym_b$, its scaling orbit leaves every compact subset as
$t\to\infty$. Properness of $H$ and Lemma~\ref{lem:radial-parametrization}
therefore give a unique $t_b(L)>1$ such that
$$
H\bigl(\delta_{t_b(L)}(j_b(L))\bigr)=s_0.
$$
Since
$$
\frac{\partial}{\partial\log t}H(\delta_t x)=|\xi|^2>0
$$
along the orbit, the implicit function theorem shows that $t_b$ is smooth. Set
$$
\iota_b(L):=\delta_{t_b(L)}(j_b(L))\in P.
$$
Hausel's description of the fiber over $[b]$
\cite[Lemma~6.10]{Hausel1998} shows that the composite of $\iota_b$ with the
quotient map $\pi_{s_0}:P\to Z$ is the Kummer quotient
$\pi_b:\Prym_b\to K_b$. Restricting \eqref{eq:normal-form} to $P$ gives
$\omega_I|_P=\pi_{s_0}^*\omega_{s_0}$, and hence
$$
\omega_{\Prym,b}
=
\pi_b^*(\omega_{s_0}|_{K_b})
=
\iota_b^*\omega_I.
$$
The map
$$
(L,u)\longmapsto
\delta_{\exp(u\log t_b(L))}(j_b(L)),
\qquad 0\leq u\leq1,
$$
is a homotopy from $j_b$ to $\iota_b$. Since $\omega_I$ is closed,
$$
[\omega_{\Prym,b}]=[j_b^*\omega_I].
$$

We now compute the periods of the right-hand side.  After choosing the
origin above, the unitary Prym torus is
$$
\frac{H^1(S_b;\ii\mathbb R)^-}
     {2\pi\ii H^1(S_b;\mathbb Z)^-}.
$$
For $u,v\in H^1(S_b;\mathbb Z)^-$, let $T_{u,v}$ be the integral two-torus
spanned by $2\pi\ii u$ and $2\pi\ii v$. Spectral dilation
$(x,\lambda)\mapsto(x,t\lambda)$ identifies the anti-invariant integral
lattices of $S_b$ and $S_{t^2b}$. Set
$$
j_{b,t}:=\delta_t\circ j_b.
$$
The maps $j_{b,t}$ are homotopic to $j_b$. Since $\omega_I$ is closed,
its period on $T_{u,v}$ is constant along the scaling ray.
Spectral dilation
preserves the semiflat norms of the two vertical tangent fields.  Both
metrics are Kähler with respect to the same complex structure $I$, so
Proposition~\ref{prop:compact-comparison} also gives
$$
|\omega_I-\omega_{\sfm}|_{g_{\sfm}}
\leq C\e^{-ct}
$$
on this compact family.  Consequently,
$$
\int_{T_{u,v}}j_{b,t}^*\omega_I
=
\int_{T_{u,v}}j_{b,t}^*\omega_{\sfm}
+O(\e^{-ct}).
$$
For completeness, we record the normalization of the vertical semiflat
form.  Let
$$
D=\begin{pmatrix}1&0\\0&-1\end{pmatrix}.
$$
For anti-invariant real harmonic one-forms $\alpha,\beta$, let
$a_\alpha,a_\beta$ be the corresponding vertical unitary connection
variations.  In the spectral splitting, their pullbacks to $S_b$ are
$\ii\alpha D$ and $\ii\beta D$.  Since $p_b:S_b\to X$ has degree two and
$\tr(D^2)=2$,
$$
-\int_X\tr(a_\alpha\wedge a_\beta)
=-\frac12\int_{S_b}\tr(\ii\alpha D\wedge\ii\beta D)
=\int_{S_b}\alpha\wedge\beta.
$$
Under the semiflat identification described above, the vertical Kähler
form is the $L^2$ pairing of these harmonic connection variations.  Thus
$$
\omega_{\sfm}(a_\alpha,a_\beta)
=
\int_{S_b}\alpha\wedge\beta.
$$
Letting $t\to\infty$ gives
$$
\int_{T_{u,v}}j_b^*\omega_I
=
(2\pi)^2\int_{S_b}u\wedge v.
$$
The harmonic representatives are smooth on the compact spectral curve, so no
local correction terms arise at the ramification points.  The restriction of the
intersection pairing to the Prym integral lattice is the Riemann form of
$\Theta_b$. Integral two-tori generate the second homology of the torus. Hence
$$
[j_b^*\omega_I]=4\pi^2\Theta_b.
$$

For a smooth double cover ramified at $2r$ points, the restricted principal
polarization has type
$$
(1^{\,r-1},2^{\,g});
$$
see \cite{Mumford1974,Beauville1977}. Here the ramification divisor is the
zero divisor of $b$, so $2r=4g-4$ and $r=2g-2$. The type is therefore
$(1^{\,2g-3},2^{\,g})$, and
$$
(2g-3)+g=3g-3=k.
$$
A polarization of type $(d_1,\ldots,d_k)$ satisfies
$$
\int_{\Prym_b}\frac{\Theta_b^k}{k!}
=
d_1\cdots d_k.
$$
Substitution of the preceding type gives
$$
\int_{\Prym_b}\frac{\Theta_b^k}{k!}=2^g.
$$
\end{proof}

Substituting the polarization class gives the volume explicitly.

\begin{corollary}\label{cor:explicit-coefficients}
With the normalization used in this paper,
$$
\Vol_{\omega_{\Prym,b}}(\Prym_b)
=
2^g(4\pi^2)^k.
$$
Consequently,
$$
\cH
=
\frac{2^{g+2k}\pi^{3k}}{k!}
=
\frac{2^{7g-6}\pi^{9g-9}}{(3g-3)!}.
$$
Moreover,
$$
\cball
=2^{-k}\cH
=
\frac{2^{4g-3}\pi^{9g-9}}{(3g-3)!}.
$$
\end{corollary}

\begin{remark}
If $g_{L^2}$ is multiplied by $c>0$, then $\omega_I$ and $H$ are multiplied
by $c$.  The sublevel and ball coefficients are multiplied by $c^k$, while
$e=c_1(L_Z)$ and the slope in \eqref{eq:DH-linear} are unchanged.  Thus
\eqref{eq:CH-intersection} and \eqref{eq:ball-main} transform consistently.
\end{remark}
\section{Geodesic balls and weighted volumes}\label{sec:balls-weighted}

\subsection{Geodesic balls}

The $1$-Lipschitz estimate \eqref{eq:lipschitz-global} gives
$$
B_{L^2}(p_0,R)
\subset
\left\{
H\leq
\frac12\bigl(R+\sqrt{2H(p_0)}\bigr)^2
\right\}.
$$
Applying \eqref{eq:sublevel-main}, we obtain
$$
\begin{aligned}
\Vball_{p_0}(R)
&\leq
A\left(
\frac12\bigl(R+\sqrt{2H(p_0)}\bigr)^2
\right)
\\
&=
2^{-k}\cH R^{2k}+O(R^{2k-1}).
\end{aligned}
$$
Therefore
\begin{equation}\label{eq:ball-limsup}
\limsup_{R\to\infty}
R^{-2k}\Vball_{p_0}(R)
\leq2^{-k}\cH.
\end{equation}

For the lower bound, fix a compact subset $\mathcal K\Subset P'$. Its distance from $p_0$ is
bounded, so \eqref{eq:length-ray-asymptotic} gives a constant
$C_{\mathcal K}>0$ such that
$$
d_{L^2}(p_0,\delta_t p)
\leq\sqrt{2H(\delta_t p)}+C_{\mathcal K}
\qquad(p\in\mathcal K,\ t\geq1).
$$
For $R>C_{\mathcal K}$,
$$
\left\{
\delta_t p:
 p\in\mathcal K,
\quad
 t\geq1,
\quad
H(\delta_t p)<\frac12(R-C_{\mathcal K})^2
\right\}
\subset B_{L^2}(p_0,R).
$$

Applying \eqref{eq:radial-volume} and the asymptotic
\eqref{eq:H-ray-asymptotic}, uniformly on the fixed compact set, to the set
on the left gives
$$
\Vol_{g_{L^2}}
\left\{
\delta_t p:
 p\in\mathcal K,
\quad t\geq1,
\quad H(\delta_t p)<s
\right\}
=
\cH(\mathcal K)s^k+o_{\mathcal K}(s^k).
$$
Taking $s=\frac12(R-C_{\mathcal K})^2$ yields
$$
\liminf_{R\to\infty}
R^{-2k}\Vball_{p_0}(R)
\geq
2^{-k}\cH(\mathcal K).
$$
Choose a compact exhaustion $\mathcal K_j\Subset P'$. Equation~\eqref{eq:CHK-exhaust}
and the limit $j\to\infty$ give
$$
\liminf_{R\to\infty}
R^{-2k}\Vball_{p_0}(R)
\geq
2^{-k}\cH.
$$
Together with \eqref{eq:ball-limsup}, this proves \eqref{eq:ball-main}.

The center $p_0$ was arbitrary, so the leading coefficient is the same for
every center.

\subsection{Exponentially weighted volume}

Pushforward by $H$ and integration by parts give
$$
V_H(\varepsilon)
=\int_0^\infty\e^{-\varepsilon s}\,\dd A(s)
=\varepsilon\int_0^\infty\e^{-\varepsilon s}A(s)\,\dd s.
$$
The boundary terms vanish: $H^{-1}(0)$ has zero ambient volume, so $A(0)=0$,
and $A(s)$ has polynomial growth.
Using \eqref{eq:sublevel-main} and $u=\varepsilon s$, we obtain
$$
V_H(\varepsilon)
=\cH\varepsilon^{-k}\int_0^\infty\e^{-u}u^k\,\dd u
+O(\varepsilon^{-k+1})
=k!\cH\varepsilon^{-k}+O(\varepsilon^{-k+1}).
$$
This proves \eqref{eq:weighted-main} and completes the proof of
Theorem~\ref{thm:main}. Since $A(s)$ is eventually polynomial, the weighted
volume has a Laurent expansion at $\varepsilon=0$. Writing
$\Coeff_{\varepsilon^{-k}}$ for its leading coefficient, we have
\begin{equation}\label{eq:weighted-coefficient}
\Coeff_{\varepsilon^{-k}}V_H(\varepsilon)
=k!\cH
=\pi^k\Vol_{\omega_{\Prym,b}}(\Prym_b).
\end{equation}
The next section obtains this coefficient from equivariant localization.
\section{Equivariant localization}\label{sec:localization}

The fixed locus $\M^{S^1}=\bigsqcup_\alpha F_\alpha$ is contained in the
nilpotent cone, since $\Hit(x)=\e^{2\ii\vartheta}\Hit(x)$ at a fixed point.
Thus its connected components are compact and finite in number.  Write
$d_\alpha=H|_{F_\alpha}$ and let $N_\alpha$ be the complex normal bundle of
$F_\alpha$.

\begin{definition}\label{def:equivariant-convention}
Regard $\varepsilon$ as a formal parameter of cohomological degree two and set
$$
\nu:=\frac{\varepsilon}{2\pi},
\qquad
d_{S^1}:=d-\varepsilon\iota_K.
$$
For a complex line bundle of circle weight $m$, set
$$
\Eul_{S^1}(L;\varepsilon)=c_1(L)+m\nu.
$$
The class $\omega_I-\varepsilon H$ is equivariantly closed.  Fixed-point
integrals denote their top-degree components in localized equivariant
cohomology.  The integral generator $u=c_1(\mathbb C_1)$ in
\cite{CHS2020} is identified with $\nu$.
\end{definition}

The polynomial volume bound allows the compact localization formula to
pass to the noncompact space.

\begin{proposition}\label{prop:noncompact-localization}
For every $\varepsilon>0$,
$$
V_H(\varepsilon)
=
\sum_\alpha
\e^{-\varepsilon d_\alpha}
\int_{F_\alpha}
\frac{\e^{\omega_I|_{F_\alpha}}}
{\Eul_{S^1}(N_\alpha;\varepsilon)}.
$$
The sum is taken before its Laurent expansion; individual summands may
have poles of order greater than $k$.
\end{proposition}

\begin{proof}
Choose a regular value $r$ above
all critical values of $H$.  The symplectic cut \cite{Lerman1995} is
$$
\overline{\M}_r
=
\left\{
(x,z)\in\M\times\mathbb C:
H(x)+\frac{|z|^2}{2}=r
\right\}\big/S^1,
$$
where the diagonal action is
$$
\e^{\ii\theta}\cdot(x,z)
=
(\e^{\ii\theta}\cdot x,\e^{\ii\theta}z).
$$
The cut contains $\{H<r\}$ as an open dense subset and has divisor
$$
Z_r=H^{-1}(r)/S^1.
$$
Since $H$ is proper, $\overline{\M}_r$ is a compact orbifold.

Set $Y_r=H^{-1}(r)$ and choose a connection $\vartheta$ on
$Y_r\to Z_r$ satisfying
$$
\vartheta(K)=1,
\qquad
\int_{S^1}\vartheta=2\pi.
$$
The normal orbifold line bundle of the cut divisor is
$$
N_r
:=
(Y_r\times\mathbb C)/{\sim},
\qquad
(y,z)\sim
(y\cdot\e^{\ii\theta},\e^{\ii\theta}z).
$$
This uses the associated line bundle convention of
\eqref{eq:euler-class-definition}.
Under the identification of the reduced spaces, $N_r$ corresponds to $L_Z$.
Its induced unitary connection has curvature $-\ii\dd\vartheta$, so the
convention $c_1(L)=[\ii F_\nabla/(2\pi)]$ gives
$$
c_1(N_r)=\left[\frac{\dd\vartheta}{2\pi}\right]=e.
$$
The residual circle acts by
$$
\e^{\ii\phi}\cdot[x,z]
=
[\e^{\ii\phi}\cdot x,z].
$$
It fixes $Z_r$, and near the divisor
$$
[\e^{\ii\phi}\cdot y,z]
=
[y,\e^{-\ii\phi}z].
$$
The residual moment map takes the constant value $r$ on $Z_r$.  The normal
weight is $-1$, and
$$
\Eul_{S^1}
\bigl(N_{Z_r/\overline{\M}_r};\varepsilon\bigr)
=
c_1(N_r)-\nu
=
c_1(N_r)-\frac{\varepsilon}{2\pi}.
$$

Apply the orbifold ABBV formula \cite{AtiyahBott1984,BerlineVergne1982} to
$\exp(\overline\omega_r-\varepsilon\overline H_r)$.  The fixed set consists
of the original fixed components $F_\alpha$ and the cut divisor.  The
contribution of each $F_\alpha$ is exactly its contribution in the original
space, because the cut does not change a neighborhood of $F_\alpha$.  The
remaining term is
$$
\mathcal R_r(\varepsilon)
  =\e^{-\varepsilon r}
    \int_{Z_r}
      \frac{\e^{\omega_r}}{c_1(N_r)-\varepsilon/(2\pi)}.
$$
Above the largest critical value, the reduced spaces are identified.  The
Duistermaat--Heckman formula \eqref{eq:DH-linear} makes $[\omega_r]$ affine in $r$.  Expanding the
denominator to the degree allowed by $\dim Z_r$ gives, for fixed
$\varepsilon>0$,
$$
\mathcal R_r(\varepsilon)
   =O_\varepsilon\bigl(\e^{-\varepsilon r}(1+r)^{2k-1}\bigr).
$$
Hence $\mathcal R_r(\varepsilon)\to0$.

The cut divisor has real codimension two and hence zero measure for the
top-degree symplectic volume form.  Therefore the top-degree integral on the
cut equals the integral over its open dense part:
$$
\int_{\overline{\M}_r}
   \e^{\overline\omega_r-\varepsilon\overline H_r}
 =\int_{\{H<r\}}\e^{-\varepsilon H}\frac{\omega_I^{2k}}{(2k)!}.
$$
Since the integrand on the right is nonnegative, the truncated weighted
volumes converge monotonically to $V_H(\varepsilon)$ as $r\to\infty$.
Taking the limit in the ABBV formula proves the asserted fixed-point
expansion.  The same argument applies to a finite quotient stack, using
orbifold integration and orbifold equivariant Euler classes.
\end{proof}

Together with \eqref{eq:weighted-coefficient}, this gives
$$
\cball
=\frac{1}{2^k k!}
\Coeff_{\varepsilon^{-k}}
\left[
\sum_\alpha\e^{-\varepsilon d_\alpha}
\int_{F_\alpha}
\frac{\e^{\omega_I|_{F_\alpha}}}
{\Eul_{S^1}(N_\alpha;\varepsilon)}
\right].
$$
We now evaluate the coefficient by the residue formula of \cite{CHS2020}.

\subsection{The odd-degree quotient and its equivariant class}

Set
$$
\Gamma_2=\operatorname{Pic}^0(X)[2],\qquad |\Gamma_2|=2^{2g},
$$
and write
$$
\pi_\Gamma:\M\longrightarrow
\M_{PGL_2}^{(1)}=[\M/\Gamma_2]
$$
for the odd-degree $PGL_2$ quotient \cite{HauselThaddeus2003Mirror}.
The action tensors a pair by a line bundle $L\in\Gamma_2$:
$$
(E,\Phi)\longmapsto(E\otimes L,\Phi\otimes1).
$$
Since $L^{\otimes2}$ is trivial, this preserves the determinant and the
trace-free condition.  Equip $L$ with its unitary flat connection.  Tensoring
the harmonic metric with a flat metric on $L$ preserves the Hitchin
equations.  In local unitary parallel frames, the endomorphism-valued
variations and their trace pairings are unchanged.  Thus the action
preserves $g_{L^2}$, $\omega_I$ and $H$, and commutes with Higgs-field
scaling.

The K\"ahler form $\omega_I$ and the Hamiltonian $H$ descend to
$\bar\omega_I$ and $\bar H$.  Define
$$
V_{PGL_2}(\varepsilon)
:=\int_{\M_{PGL_2}^{(1)}}\exp(\bar\omega_I-\varepsilon\bar H).
$$
Integration over a finite quotient stack divides the integral of the
pulled-back form by the group order.  Consequently,
\begin{equation}\label{eq:MNS-volume-quotient}
V_{PGL_2}(\varepsilon)
=2^{-2g}V_H(\varepsilon),
\end{equation}
with finite stabilizers already accounted for by orbifold integration.
Combining this equality with \eqref{eq:weighted-coefficient} and
$\cball=2^{-k}\cH$ gives
\begin{equation}\label{eq:MNS-ball-transform}
\cball
=\frac{2^{2g-k}}{k!}
\Coeff_{\varepsilon^{-k}}V_{PGL_2}(\varepsilon).
\end{equation}
This quotient is the $p=1\in\mathbb Z_2$ sector in
\cite[Section~5.1]{MNS2000}, after removal of the central $U(1)$ factor.
Our comparison concerns the finite-dimensional equivariant integral in
the normalization fixed here.

Let $\End\mathbb E$ be the endomorphism bundle associated with the universal
projective bundle over $X\times\M_{PGL_2}^{(1)}$.  For the positive generator
$\omega_X\in H^2(X;\mathbb Z)$ with $\int_X\omega_X=1$, the class
$\alpha\in H^2_{\mathbb C^*}(\M_{PGL_2}^{(1)})$ of \cite{CHS2020} is defined by
\begin{equation}\label{eq:CHS-alpha-definition}
c_2(\End\mathbb E)
=2\alpha\otimes\omega_X+\text{other K\"unneth components}.
\end{equation}
We use $\alpha$ also for its image in ordinary cohomology.  The fixed locus
of the quotient consists of the stable-bundle moduli space $F_0$ and
$$
F_i\simeq\operatorname{Sym}^{2g-2i-1}(X),
\qquad 1\leq i\leq g-1;
$$
see \cite{Hitchin1987,Thaddeus1994,HauselThaddeus2003Relations}.
The following calculation fixes the class and the equivariant parameter
needed in the residue formula.

\begin{lemma}\label{lem:alpha-kahler-normalization}
In ordinary cohomology,
\begin{equation}\label{eq:hausel-normalization-dictionary}
[\omega_I]=2\pi^2\pi_\Gamma^*\alpha.
\end{equation}
Moreover, $\bar H|_{F_0}=0$ and
\begin{equation}\label{eq:fixed-H-value}
\bar H|_{F_i}=\pi(2i-1),\qquad 1\leq i\leq g-1.
\end{equation}
With
\begin{equation}\label{eq:a-u-substitution}
a=2\pi^2,\qquad \nu=\frac{\varepsilon}{2\pi},
\end{equation}
one has $\pi_\Gamma^*(a\alpha)=[\omega_I-\varepsilon H]$ in localized
$S^1$-equivariant cohomology.
\end{lemma}

\begin{proof}
We first compute the ordinary class.  Since
$c_1(\End\mathbb E)=0$, the unitary Chern--Weil convention used here
gives
$$
c_2(\End\mathbb E)
=
\frac{1}{8\pi^2}
\Tr_{\End}(\mathbb F\wedge\mathbb F).
$$
On the component defined by $\Phi=0$, decompose the universal curvature
into its components along $X$
and along the moduli space.  If $a_1,a_2$ are two fixed-determinant
connection variations, then the value of the $(2,2)$ Künneth component
of $\mathbb F\wedge\mathbb F$ on $(a_1,a_2)$ is
$$
-2\Tr_{\End}(\operatorname{ad}a_1\wedge\operatorname{ad}a_2).
$$
The sign comes from interchanging the two moduli one-forms, and the
factor $2$ comes from the two mixed terms.  For $\mathfrak{sl}_2$, the
adjoint and fundamental traces are related by
$$
\Tr_{\End}(\operatorname{ad}U\,\operatorname{ad}V)
=4\tr(UV).
$$
The slant product with $[X]$ therefore gives
$$
\begin{aligned}
\left(c_2(\End\mathbb E)/[X]\right)(a_1,a_2)
&=
-\frac{1}{4\pi^2}
\int_X
\Tr_{\End}(\operatorname{ad}a_1\wedge\operatorname{ad}a_2)
\\
&=
-\frac1{\pi^2}\int_X\tr(a_1\wedge a_2)
=
\frac1{\pi^2}\omega_I(a_1,a_2).
\end{aligned}
$$
Equation~\eqref{eq:CHS-alpha-definition} states that
$c_2(\End\mathbb E)/[X]=2\alpha$.  Hence
$$
2\pi_\Gamma^*\alpha=\frac1{\pi^2}[\omega_I].
$$
The degree-two ordinary cohomology of the odd $PGL_2$ component is generated
by the universal class $\alpha$, whose restriction to $F_0$ is nonzero; see \cite[Section~1.2]{CHS2020}.  Thus the
scalar relating $\pi_\Gamma^*\alpha$ and $[\omega_I]$ is determined by the
preceding universal-connection calculation on the component defined by $\Phi=0$.
This proves the ordinary cohomology identity in
\eqref{eq:hausel-normalization-dictionary}.  This is the rank-two case of the
Atiyah--Bott universal-connection calculation; see
\cite{AtiyahBott1983,HauselThaddeus2004Generators,CHS2020}.

We next compute $H$ on the inverse image of $F_i$ in $\M$. There,
$$
E=L\oplus M,
\qquad
M=\Lambda L^{-1},
\qquad
\deg L=i,
\qquad
\deg M=1-i,
$$
and the Higgs field may be written as
$$
\Phi=
\begin{pmatrix}
0&0\\
\varphi&0
\end{pmatrix},
\qquad
\varphi:L\longrightarrow MK_X.
$$
With respect to the harmonic metric,
$$
[\Phi\wedge\Phi^\dagger]
=
\begin{pmatrix}
-\varphi\wedge\varphi^\dagger&0\\
0&\varphi\wedge\varphi^\dagger
\end{pmatrix}.
$$
The two diagonal Hitchin equations have the same central curvature term.
Their difference is
$$
F_L-F_M-2\varphi\wedge\varphi^\dagger=0.
$$
Our $L^2$ convention gives
$$
\|\Phi\|_{L^2}^2
=2\ii\int_X\varphi\wedge\varphi^\dagger.
$$
Consequently,
$$
\begin{aligned}
\|\Phi\|_{L^2}^2
&=\ii\int_X(F_L-F_M)
\\
&=2\pi\bigl(\deg L-\deg M\bigr)
\\
&=2\pi(2i-1),
\end{aligned}
$$
where the second equality uses
$\deg L=(\ii/2\pi)\int_X F_L$.  It follows that
$$
H=\frac12\|\Phi\|_{L^2}^2=\pi(2i-1).
$$
Since $H=\pi_\Gamma^*\bar H$, this proves
\eqref{eq:fixed-H-value} and agrees with the
Morse-theoretic formula for the Hitchin function on the rank-two fixed
components; see \cite{Hitchin1987,Thaddeus1994}.

By \cite[Lemma~2 in the proof of Theorem~3.3]{CHS2020}, the coefficient
of $\nu$ in $\alpha|_{F_i}$ is $-(2i-1)$, and this coefficient vanishes on
$F_0$.  The ordinary part is already determined by
\eqref{eq:hausel-normalization-dictionary}.  In degree two, these are the
only two components because the circle acts trivially on each $F_i$.  Under
\eqref{eq:a-u-substitution}, the equivariant part becomes
$$
-a(2i-1)\nu
=-\varepsilon\pi(2i-1)
=-\varepsilon\bar H|_{F_i}.
$$
The restrictions of $\pi_\Gamma^*(a\alpha)$ and
$[\omega_I-\varepsilon H]$ therefore agree on every fixed component.
Injectivity of fixed-point restriction after inverting the equivariant
parameter proves the localized identity.
\end{proof}

\subsection{The residue formula and its leading term}

Write
$$
I_g(a,\nu):=\oint_{\M_{PGL_2}^{(1)}}\exp(a\alpha),
$$
where $\oint$ is the equivariant integral of \cite{CHS2020}.  The preceding
lemma and Proposition~\ref{prop:noncompact-localization} identify it with
the convergent weighted volume:
$$
V_{PGL_2}(\varepsilon)
=I_g\left(2\pi^2,\frac{\varepsilon}{2\pi}\right).
$$
Specializing \cite[Theorem~3.3 and equation~(3.14)]{CHS2020} to
$\exp(a\alpha)$ gives
\begin{equation}\label{eq:CHS-residue-sum}
I_g(a,\nu)
=\sum_{r\in\{0,\nu,-\nu\}}\Res_{y=r}\Omega_{a,\nu}(y),
\end{equation}
where
\begin{equation}\label{eq:CHS-Omega}
\Omega_{a,\nu}(y)
=\frac14
\frac{\bigl[-a(\nu^2-y^2)-2\nu\bigr]^g\,\dd y}
{\begin{aligned}
&\nu^{g-1}\sinh(-ay)
 \bigl(\nu+y\tanh(ay/2)\bigr)\\
&\qquad\times\bigl(\nu+y\coth(ay/2)\bigr)
 y^{2g-2}(\nu^2-y^2)^{2g-2}
\end{aligned}}.
\end{equation}
The leading term is most readily computed from the two additional poles
that approach the origin.

\begin{proposition}\label{prop:equivariant-residue-leading-term}
For fixed $a>0$, as $\nu\to0^+$,
\begin{equation}\label{eq:Ig-leading-pole}
I_g(a,\nu)=2^{-g}a^k\nu^{-k}+O(\nu^{-k+1}),
\qquad k=3g-3.
\end{equation}
\end{proposition}

\begin{proof}
Set
$$
f_\nu(y):=\nu+y\tanh(ay/2),
\qquad
g_\nu(y):=\nu+y\coth(ay/2).
$$
Since $f_\nu$ is even, write $z=y^2$.  Taylor expansion gives
$$
f_\nu(y)
=\nu+\frac a2z+O(z^2).
$$
The implicit function theorem applied to $f_\nu=0$ at
$(z,\nu)=(0,0)$ gives a unique small solution
$$
z(\nu)=-\frac{2\nu}{a}+O(\nu^2).
$$
Its two square roots give simple zeros $b_\pm(\nu)$ satisfying
\begin{equation}\label{eq:bethe-small-roots}
b_\pm^2
=-\frac{2\nu}{a}+O(\nu^2),
\qquad
b_-(\nu)=-b_+(\nu).
\end{equation}
Their simplicity also follows from
$$
f_\nu'(b)
=
\tanh(ab/2)
+\frac{ab}{2}\operatorname{sech}^2(ab/2)
=ab\,(1+O(\nu)),
$$
and $b\neq0$ for $\nu\neq0$.

By \cite[Proposition~3.5]{CHS2020}, there is a small circle
$|y|=\rho$, independent of $\nu$, that avoids the zeros of
$\sinh(-ay)$, $f_\nu(y)$, and $g_\nu(y)$.  For sufficiently small positive
$\nu$, the only additional poles of the meromorphic one-form
$\Omega_{a,\nu}$ inside this circle, besides $0,\pm\nu$, are $b_\pm$.  Applying the residue theorem to all poles
inside the circle and identifying the residues at $0,\pm\nu$ with
$I_g(a,\nu)$ through \eqref{eq:CHS-residue-sum} gives
\begin{equation}\label{eq:residue-decomposition-small}
I_g(a,\nu)
=
\frac1{2\pi\ii}\int_{|y|=\rho}\Omega_{a,\nu}(y)
-
\sum_{\sigma\in\{+,-\}}
\Res_{y=b_\sigma}\Omega_{a,\nu}(y).
\end{equation}
On the fixed circle, the functions $y$, $\sinh(-ay)$, $f_\nu(y)$,
$g_\nu(y)$, and $\nu^2-y^2$ are uniformly bounded away from zero, and
the numerator is uniformly bounded.  The only explicit singular factor
is $\nu^{-(g-1)}$.  Hence
\begin{equation}\label{eq:outer-circle-order}
\frac1{2\pi\ii}\int_{|y|=\rho}\Omega_{a,\nu}(y)
=O(\nu^{1-g}).
\end{equation}
Since $k=3g-3$ and $g\geq2$, one has
$1-g\geq-k+1$, so this term is $O(\nu^{-k+1})$.

Fix $b=b_\pm$.  Since $f_\nu$ is the only factor vanishing at $b$, the
simple-pole formula gives
\begin{equation}\label{eq:small-pole-residue-exact}
\Res_{y=b}\Omega_{a,\nu}
=
\frac14
\frac{[-a(\nu^2-b^2)-2\nu]^g}
{\nu^{g-1}\sinh(-ab)\,g_\nu(b)\,
 b^{2g-2}(\nu^2-b^2)^{2g-2}f_\nu'(b)}.
\end{equation}
Equation~\eqref{eq:bethe-small-roots} gives the termwise expansions
$$
\begin{aligned}
-a(\nu^2-b^2)-2\nu
&=-4\nu+O(\nu^2),\\
\sinh(-ab)
&=-ab\,(1+O(\nu)),\\
g_\nu(b)=\nu+b\coth(ab/2)
&=\frac2a+O(\nu),\\
b^{2g-2}
&=\left(-\frac{2\nu}{a}\right)^{g-1}(1+O(\nu)),\\
(\nu^2-b^2)^{2g-2}
&=\left(\frac{2\nu}{a}\right)^{2g-2}(1+O(\nu)),\\
f_\nu'(b)
&=ab\,(1+O(\nu)).
\end{aligned}
$$
Substituting these expansions into
\eqref{eq:small-pole-residue-exact} gives
$$
\frac{\frac14(-4\nu)^g}
{4(-1)^{g-1}2^{3g-3}a^{-(3g-3)}\nu^{4g-3}}
=-2^{-g-1}a^k\nu^{-k}
$$
for the leading term. Consequently,
\begin{equation}\label{eq:small-pole-residue-leading}
\Res_{y=b}\Omega_{a,\nu}
=
-2^{-g-1}a^{3g-3}\nu^{-3g+3}(1+O(\nu))
=
-2^{-g-1}a^k\nu^{-k}(1+O(\nu)).
\end{equation}
The two roots have the same leading residue.  Substituting
\eqref{eq:small-pole-residue-leading} and
\eqref{eq:outer-circle-order} into
\eqref{eq:residue-decomposition-small} yields
$$
I_g(a,\nu)
=
2^{-g}a^k\nu^{-k}+O(\nu^{-k+1}),
$$
which is \eqref{eq:Ig-leading-pole}.
\end{proof}

Substituting \eqref{eq:a-u-substitution} into \eqref{eq:Ig-leading-pole}
gives
$$
V_{PGL_2}(\varepsilon)
=2^{5g-6}\pi^{9g-9}\varepsilon^{-(3g-3)}
+O(\varepsilon^{-(3g-4)}),
$$
since $2^{-g}(2\pi^2)^k(2\pi)^k=2^{5g-6}\pi^{9g-9}$.
By \eqref{eq:MNS-volume-quotient} and \eqref{eq:MNS-ball-transform}, this
independently recovers the coefficients of
Corollary~\ref{cor:explicit-coefficients}.

The additional poles in \eqref{eq:CHS-Omega} are the Bethe roots appearing
in \cite{MNS2000}; applying the residue theorem recovers the corresponding
infinite sum \cite[Remark~3.6]{CHS2020}.  The formulas above concern the
odd $PGL_2$ sector.  Allowing arbitrary determinant and trace adds, up to a
finite quotient, a $T^*\Jac(X)$ factor and its abelian contribution to the
weighted volume.

\end{document}